\documentclass[a4paper,12pt,onecolumn]{article}

\usepackage[vmargin=2cm,hmargin=2cm,headheight=14.5pt,top=2cm,headsep=.5cm]{geometry}
\usepackage[utf8]{inputenc}
\usepackage{bm}
\usepackage{empheq}
\usepackage{stackrel}
\usepackage{cases}
\usepackage{mathtools}
\usepackage{amsthm,amsmath,amscd}
\usepackage{makeidx}
\usepackage[charter]{mathdesign}
\usepackage{tikz-cd}
\tikzcdset{every label/.append style = {font = \small}}

\usepackage{graphicx}
\DeclareMathSizes{12}{12}{8}{6}

\usepackage{cite}
\usepackage{url}
\usepackage[ddmmyy]{datetime}

\usepackage{marvosym}
\newcommand*{\footnotemarkcolor}{red}
\makeatletter
\renewcommand*{\@makefnmark}{\hbox{\@textsuperscript{%
 \color{\footnotemarkcolor}\normalfont\@thefnmark}}}
\makeatother
\makeatletter
\def\@fnsymbol#1{\ensuremath{\ifcase#1\or \text{\Mercury}
\or \text{\Venus} \or \text{\Earth} \or
 \text{\Jupiter} \or \text{\Saturn} \or \text{\Neptune} \or \text{\Uranus} \or \text{\Pluto}
 \or \text{\Moon} \or \text{\Sun}\or \text{\Mercury}
 \or \text{\Venus} \or \text{\Earth} \or
 \text{\Jupiter} \or \text{\Saturn} \or \text{\Neptune} \or \text{\Uranus} \or \text{\Pluto}
 \or \text{\Moon} \or \text{\Sun}\or \text{\Mercury}
 \or \text{\Venus} \or \text{\Earth} \or
 \text{\Jupiter} \or \text{\Saturn} \or \text{\Neptune} \or \text{\Uranus} \or \text{\Pluto}
 \or \text{\Moon} \or \text{\Sun}
\else\@ctrerr\fi}}%
\makeatother

\newtheoremstyle{ptheorem}{1em}{0em}{\itshape}{}{\bfseries}{.}{.5em}{\thmname{#1}\thmnumber{ #2}\thmnote{ (\hspace{-.01pt}{#3})}}

\theoremstyle{ptheorem}

\newtheorem{thm}{Theorem}[section]

\newtheorem{lem}[thm]{Lemma}

\newtheoremstyle{hdef}{1em}{0em}{}{}{\bfseries}{.}{.5em}{\thmname{#1}\thmnumber{ #2}\thmnote{ (\hspace{-.01pt}{#3})}}
\theoremstyle{hdef}

\newtheorem{dfn}[thm]{Definition}
\newtheorem{rem}[thm]{Remark}

\makeatletter
\newtheoremstyle{premark}{1em}{0em}{
\addtolength{\@totalleftmargin}{1.5em}
\addtolength{\linewidth}{-1.5em}
\parshape 1 1.5em \linewidth}{}{\scshape}{.}{.5em}{}
\makeatother

\theoremstyle{premark}

\numberwithin{equation}{section}
\numberwithin{figure}{section}

\DeclareMathOperator{\Id}{Id}

\DeclareMathOperator{\dif}{d}

\newcommand{\cC}{{\mathcal C}}

\newcommand{\cL}{{\mathcal L}}

\newcommand{\bN}{{\mathbb N}}

\newcommand{\bR}{{\mathbb R}}

\renewcommand{\a}{\alpha}

\renewcommand{\l}{\lambda}

\newcommand{\ol}{\overline}

\renewcommand{\(}{\left(}
\renewcommand{\)}{\right)}

\newcommand{\til}{\widetilde}
\newcommand{\Lsp}[1]{\operatorname{L^{#1}}}

\allowdisplaybreaks

\begin{document}
\title{Existence of solutions of integral equations defined in unbounded domains via spectral theory\footnote{The three authors were partially supported by Ministerio de Econom\'ia y Competitividad, Spain, and FEDER, project MTM2013-43014-P, and by the Agencia Estatal de Investigaci\'on (AEI) of Spain under grant MTM2016-75140-P, co-financed by the European Community fund FEDER.}}

\author{
Alberto Cabada\\
\normalsize e-mail: alberto.cabada@usc.es\\
Luc\'ia L\'opez-Somoza\footnote{Supported by  FPU scholarship, Ministerio de Educaci\'on, Cultura y Deporte, Spain.} \\
\normalsize e-mail: lucia.lopez.somoza@usc.es\\
F. Adri\'an F. Tojo \\
\normalsize e-mail: fernandoadrian.fernandez@usc.es\\
\normalsize \emph{Instituto de Ma\-te\-m\'a\-ti\-cas, Facultade de Matem\'aticas,} \\ \normalsize\emph{Universidade de Santiago de Com\-pos\-te\-la, Spain.}\\ 
}
\date{}

\maketitle


\begin{abstract}
In this work we study integral equations defined on the whole real line. Using a suitable Banach space, we look for solutions which satisfy some certain kind of asymptotic behavior. We will consider spectral theory in order to find fixed points of the integral operator.
\end{abstract}



\section{Introduction}
In this paper we will study the existence of fixed points of the following integral operator 
\begin{equation*}
	T u(t):=\int_{-\infty}^{\infty} k(t,s)\,\eta(s)\,f(s,u(s)) \dif s.
\end{equation*}

When working with integral problems defined in unbounded intervals, the main difficulty is the lack of compactness of the operator. In the recent literature (see \cite{MinCar,MinCar2,MinCar-4ord,FiMinCar,DjeGue}), most of the authors use the following relatively compactness criterion (see \cite{Corduneanu,przerad}) to deal with this problem:

\begin{thm}[{\cite[Theorem 1]{przerad}}] \label{thm_equic}
	Let $E$ be a Banach space and $\cC(\mathbb{R},E)$ the space of all bounded continuous functions $x\colon \bR\rightarrow E$. For a set $D\subset \cC(\bR,E)$ to be relatively compact, it is necessary and sufficient that:
	\begin{enumerate}
		\item $\{x(t), \ x\in D\}$ is relatively compact in $E$ for any $t\in\bR$;
		\item for each $a>0$, the family $D_a:=\{x|_{[-a,a]}, \ x\in D \}$ is equicontinuous;
		\item $D$ is stable at $\pm\infty$, that is, for any $\varepsilon>0$, there exists $T>0$ and $\delta>0$ such that if $\|x(T)-y(T)\|\le \delta$, then $\|x(t)-y(t)\|\le \varepsilon$ for $t\ge T$ and if $\|x(-T)-y(-T)\|\le \delta$, then $\|x(t)-y(t)\|\le \varepsilon$ for $t\le -T$, where $x$ and $y$ are arbitrary functions in $D$.
	\end{enumerate}
\end{thm}

In a recent paper \cite{Somoza}, the authors presented a novel way of dealing with the problem of the lack of compactness of the integral operator. They defined a new kind of Banach space: the \textit{space of continuously $\boldmath{n}$-differentiable $\boldmath{ \varphi}$-extensions to infinity}. Moreover, this Banach space makes it possible to study the asymptotic behavior of the solutions of the problem. In that work, the authors used fixed point index methods in order to obtain existence and multiplicity results for boundary value problems and Hammerstein-type equations of the kind
\begin{equation}\label{eqthamm}
	L u(t):=p(t)+\int_{-\infty}^{\infty} k(t,s)\,\eta(s)\,f(s,u(s)) \dif s.
\end{equation}
Furthermore, those results included the location of the solutions, since they were found in a cone defined in a general abstract way --cf. \cite{FigToj}. Depending on the region of the cone, the index was to be proven zero or nonzero, thus providing a solution to the problem of study.

In this paper we complement those findings by approaching the problem in a different way. If in the previous work we had fairly restrictive conditions on the nonlinearity $f$, here we relax in a significant way those restrictions by studying the eigenvalues of some related linear operators. This approach has been used successfully previously, as we can see in the works of Infante et al. \cite{gi-pp-ft}, Webb and Lan \cite{jwkleig} or even in the case of linearly bounded nonlinear operators as shown in \cite{nlpert}. We note that, for the sake of simplicity, we do not include in this paper the function $p$ occurring in \eqref{eqthamm}. However, it could be included with minor adaptations, following the hypotheses for $p$ in \cite{Somoza}.

There is of course a price to pay for the advantage regarding the nonlinearity, and is that the conditions on the kernel $k$ occurring in \eqref{eqthamm} are more restrictive. One can check that this is the case of the kernel in the application studied in \cite{Somoza}. There, we have studied the equation that describes the movement of a self-propelled projectile launched vertically from the surface of a planet, that is, 
\begin{equation}\label{ec-rocket}
u''(t)=-\frac{g\,R^2}{(u(t)+R)^2}+h(t,u(t)), \quad t\in[0,\infty); \ u(0)=0, \ u'(0)=v_0,
\end{equation}
where $u$ represents the distance from the surface of the planet, $R$ is the radius of the planet, $g$ the surface gravity constant, $v_0$ the initial velocity and $h(t,y)$ the acceleration generated by the propulsion system of the rocket.

Rewriting previous problem as an integral one, we can see that solutions of \eqref{ec-rocket} coincide with fixed points of operator
\begin{equation*}
Lu(t)=v_0\,t+\int_{0}^{\infty}k(t,s)\,f(s,u(s))\, \dif s,
\end{equation*}
where
\[k(t,s)=\left\{\begin{array}{ll}
t-s, & 0\le s\le t, \\
0, & \text{otherwise},
\end{array} \right.\]
and \[f(t,y)=-\frac{g\,R^2}{(y+R)^2}+h(t,y).\]

We note that the results in the present work could not be applied to this problem (even with the modification to include the term $p$) as, for instance, condition $(C_2)$ does not hold. At the same time, we will show in Section 5 an example which is solved with the method developed in this paper but does not satisfy the hypotheses required in \cite{Somoza}.

Thus, our two methods are not comparable but complementary, making it possible to deal with different kinds of differential and integral problems defined on unbounded intervals, either with more restrictive conditions on the linear part or on the nonlinear one.

This paper is divided in the following way: in Section 2 we summarize some definitions and results of spectral theory. In Section 3 we compile the theory regarding the space of continuously $n$-differentiable $\varphi$-extensions to infinity, which has been developed in \cite{Somoza}. Section 4 includes our results of existence of solutions of integral problems. Finally, Section 5 shows an example to which the results in Section 4 are applied.

\section{Preliminaries}
Let $(N_1,\|\cdot\|_1)$ and $(N_2,\|\cdot\|_2)$ be two normed spaces. Let $\Gamma:N_1\to N_2$ be a bounded linear operator, that is, such that its norm $\|\Gamma\|=\sup_{\|u\|_2=1}{\|\Gamma u\|_1/\|u\|_2}$ is finite.  
We recall that $\lambda$ is an \textit{eigenvalue} of a linear operator between normed spaces $\Gamma:(N_1,\|\cdot\|_1)\to (N_2,\|\cdot\|_2)$ with corresponding eigenfunction $\phi$ if $\phi \neq 0$ and ${\lambda}\,\phi=\Gamma\, \phi$. The reciprocals of nonzero
eigenvalues are called \emph{characteristic values} of $\Gamma$. 
We will denote the \textit{spectral radius} of $\Gamma$ by $r(\Gamma):=\lim_{n\to\infty}\|\Gamma^n\|^{\frac{1}{n}}$ and its \textit{principal characteristic value}  by $\mu(\Gamma):=1/r(\Gamma)$.\par

We recall now some known definitions and results.
%
%


\begin{dfn}
	We say that $K$ is a \emph{total cone} if $\ol{K-K}=X$.
\end{dfn}

\begin{thm}[{Krein-Rutman \cite[Theorem 1.1]{Du}}]
	Let $K$  be a total cone and
	$\cL : X\to X$ a compact linear operator that maps $K$ to $K$
	with positive spectral radius $r(\cL)$. Then $r(\cL)$ is an eigenvalue with an	eigenvector $\phi\in K\backslash\{0\}$.
\end{thm}


\begin{thm}[{\cite[Theorem 2.7]{jw-tmna}}]\label{lowbound}
	Let $\cL$ be a bounded linear operator in a Banach space $X$ and let $K$ be a cone in $X$ such that $\cL(K)\subset K$. If there exists $\lambda_0>0$ and $v\in K\setminus\{0\}$ such that $Lv\curlyeqsucc_K \lambda_0 \, v$, then $r(\cL)\ge \lambda_0$.
\end{thm}

\begin{thm}[{\cite[Theorem 1]{ZabKrasSte}}]\label{upbound}
	Let the positive, completely continuous, linear operator $\cL$ satisfy the inequality
	\[\cL \,v \le\lambda_0 \,v, \]
	where $v$ is a quasi-interior element of the cone $K$. Then $r(\cL)\leq \lambda_0$.
\end{thm}

\begin{rem}\label{rem_quasi-int}
If the cone $K$ has non empty interior, then the interior and the quasi-interior of the cone coincide (see \cite{FullBraun}). 
\end{rem}


%
%

\section{The space of continuously $\boldmath{n}$-differentiable $\boldmath{ \varphi}$-extensions to infinity}
In this section, we review the concepts introduced in \cite{Somoza}.

Consider the space $\ol \bR:=[-\infty,+\infty]$ with the compact topology, that is, the topology generated by the basis
\[\{B(a,r)\ :\ a\in\bR,\ r\in\bR^+\}\cup\{[-\infty,a)\ :\ a\in\bR\}\cup\{(a,+\infty]\ :\ a\in\bR\}.\]
With this topology, $\ol \bR$ is homeomorphic to any compact interval of $\bR$ with the relative topology inherited from the usual topology of $\bR$.\par
It is easy to check that $\cC(\ol\bR,\bR)$ is a Banach space with the usual supremum norm. We define, in a similar way,
\[\cC^n(\ol\bR,\bR):=\left\{f:\ol\bR\to\bR\ :\ f|_{\bR}\in\cC^n(\bR,\bR),\ \exists \lim_{t\to\pm\infty}f^{(j)}(t) \in \bR,\ j=0,\dots,n\right\},\]
 for $n\in\bN$. $\cC^n(\ol\bR,\bR)$, $n\in\bN$, is a Banach space with the norm \[\|f\|_{(n)}:=\sup\left\{\left\|f^{(k)}\right\|_\infty\ :\ k=0,\dots,n\right\}.\]
Take now $\varphi\in\cC^n(\bR,\bR^+)$, where $\bR^+=(0,\infty)$, and define the \emph{space of continuously $n$-differentiable $ \varphi$-extensions to infinity}
\[\widetilde\cC^n_\varphi\equiv\widetilde\cC^n_\varphi(\bR,\bR)=\left\{f\in\cC^n(\bR,\bR)\ :\ \exists \til f\in\cC^n(\ol\bR,\bR),\  f=\varphi \left(\til f|_{\bR}\right)\right\}.\]
We define the norm
\[\|f\|_\varphi:=\|\til f\|_{(n)},\ f\in\widetilde\cC_\varphi.\]
$\|\cdot\|_\varphi$ is well defined, since the extension $\til f$ is unique for every $f$; indeed, assume there are $\til f_1$, $\til f_2$ such that $\til f_1\,\varphi=\til f_2\,\varphi=f$ in $\bR$. Since $\bR$ is dense in $\ol \bR$ and $\til f_1$ and $\til f_2$ are continuous, $\til f_1=\til f_2$.\par
On the other hand, for every $\til f\in\cC^n(\ol\bR,\bR)$ there exists a unique $f\in\widetilde\cC_\varphi$ such that $\til f|_{\bR}\,\varphi=f$ (just define $f:=\til f\,\varphi$ in $\bR$). 

This shows that there is an isometric isomorphism
\begin{equation*}\begin{split}
\Phi:\cC^n(\ol\bR,\bR)&\to\widetilde\cC^n_\varphi \\
\tilde{f}&\mapsto \Phi(\tilde{f})=\tilde{f}|_{\bR}\,\varphi,
\end{split}\end{equation*}
of which the inverse isomorphism is 
\begin{equation*}\begin{split}
\Phi^{-1}:\widetilde\cC^n_\varphi &\to \cC^n(\ol\bR,\bR)\\
f&\mapsto \Phi^{-1}(f)=f/\varphi .
\end{split}\end{equation*}

Furthermore, Arcel\`a-Ascoli's Theorem applies to $\cC^n(\ol\bR,\bR)$ since $\ol\bR$ is a Hausdorff compact topological space and $\bR$ is a complete metric space. Using $\Phi$ we can apply the Theorem to $\widetilde\cC^n_\varphi$. To be precise,

\begin{thm}
$F\subset \til \cC_\varphi^n$ has compact closure if and only if the two following conditions are satisfied:
\begin{itemize}
	\item  For each $t\in\bR$, the set $\{\til f(t), \ f\in F\}$ has compact closure or, which is the same (since $\til f(t) \in \bR$), $\{\til f(t), \ f\in F\}$ is bounded, that is, for each $t\in\bR$ there exists some constant $M>0$ such that
	\[\left|\frac{\partial^j \til f}{\partial t^j}(t)\right|=\left|\frac{\partial^j (f/\varphi)}{\partial t^j}(t)\right|\le M<\infty,\]
	for all $j=0,\dots,n$ and $f\in F$.
	\item $F$ is equicontinuous, that is, for all $\varepsilon\in \bR^+$ there exists some $\delta\in\bR^+$ such that 
	\[\left|\frac{\partial^j \til f}{\partial t^j}(r)- \frac{\partial^j \til f}{\partial t^j}(s)\right|=\left|\frac{\partial^j (f/\varphi)}{\partial t^j}(r)- \frac{\partial^j (f/\varphi)}{\partial t^j}(s)\right|<\varepsilon,\]
	for all $j=0,\dots,n$, $f\in F$ and $r,\,s\in\bR$ such that $|r-s|<\delta$.
\end{itemize}
\end{thm}

More properties of these spaces can be found in \cite{Somoza}.

\section{Eigenvalue criteria}\label{seceigen}

In this section we will study the existence of fixed points of an operator $T$ on $\widetilde\cC^n_\varphi$ given by equation \eqref{eqthamm}. In particular, we will look for solutions of the previous integral equation in abstract cones, which will be defined following the line of \cite{FigToj}. In that work, the authors considered a real normed space $(N,\|\cdot\|)$ and a continuous functional $\alpha\colon N\rightarrow \bR$. They proved that, when $\alpha$ satisfies the three following properties:
\begin{enumerate}
	\item[$(P_1)$] $\alpha(u+v)\ge \alpha(u)+ \alpha(v), \text{ for all } u,\,v \in N;$
	\item[$(P_2)$] $\alpha(\lambda\,u)\ge \lambda\,\alpha(u)$, for all $u\in N$, $\lambda\ge 0$;
	\item[$(P_3)$] $\left[\alpha(u)\ge 0, \ \alpha(-u)\ge 0\right] \Rightarrow u\equiv 0$; 
\end{enumerate}
then 
\[K_\alpha=\left\{u\in N\ : \ \alpha(u)\ge 0\right\}\]
is a cone.

This way, we will consider the abstract cone 
\begin{displaymath}K_\alpha=\left\{u\in\widetilde{\cC}^n_\varphi\ : \ \alpha(u)\ge 0\right\},\end{displaymath}
where $\alpha\colon \widetilde{\cC}^n_\varphi \rightarrow \bR$ is a functional satisfying $(P_1)-(P_3)$.

\begin{rem}
	If the cone $K$ is defined by a continuous functional $\alpha$ (as it will occur with the cones considered in this paper), then $v$ an element of the cone will belong to its interior if and only if $\alpha(v)>0$.
\end{rem}

In order to state our eigenvalue comparison results, we consider the following operator on $\widetilde\cC^n_\varphi$:
\[
L_1\,u(t):=  \int_{-\infty}^{\infty}| k(t,s) \, \eta(s)|\,u(s)\,\dif s.
\]

Consider $P$, the cone of nonnegative functions in  $\widetilde\cC^n_\varphi$, that is
\[P:=\left\{u\in\widetilde\cC^n_\varphi \ : \ u\ge 0 \text{ on } \bR \right\}.\]

In this section we will assume the following hypotheses:
\begin{itemize}
\item [ $(C_{1})$] The kernel $k:\bR^2\to\bR$, is such that $k(\cdot,s)\,\eta(s)\in\widetilde\cC_{\varphi}^n$ for every $s\in\bR$. Moreover, 
\begin{itemize}
	\item if $n=0$, then for every $\varepsilon >0$, there exist $\delta>0$ and a measurable function $\omega_0$ such that if $|t_1-t_2|<\delta$ then
	\begin{itemize}
		\item[(i)] \[\left|\,\frac{|k(t_1,s)\,\eta(s)|}{\varphi(t_1)}-\frac{|k(t_2,s)\,\eta(s)|}{\varphi(t_2)}\,\right| < \varepsilon\,\omega_0(s),\]
		\item[(ii)] \[\left|\,\frac{(k(t_1,s)\,\eta(s))^+}{\varphi(t_1)}-\frac{(k(t_2,s)\,\eta(s))^+}{\varphi(t_2)}\,\right| < \varepsilon\,\omega_0(s),\]
	\end{itemize} 
and
\begin{itemize}
	\item[(iii)] 		\[\left|\,\frac{k(t_1,s)\,\eta(s)}{\varphi(t_1)}-\frac{k(t_2,s)\,\eta(s)}{\varphi(t_2)}\,\right| < \varepsilon\,\omega_0(s)\]
\end{itemize}
	for a.\,e. $s \in \mathbb{R}$. Here, as usual, $\left(k(t,s) \, \eta(s)\right)^+=\max\left\{k(t,s) \, \eta(s),\, 0 \right\}$;
	\item if $n>0$, $k(t,s)\,\eta(s)\ge 0$ and for every $\varepsilon >0$ and $j=0,\dots,n$, there exist $\delta>0$ and a measurable function $\omega_j$ such that if $|t_1-t_2|<\delta$ then \[\left|\frac{\partial ^j (k(\cdot,s)\,\eta(s)/\varphi(\cdot))}{\partial t^j}(t_1)-\frac{\partial ^j (k(\cdot,s)\,\eta(s)/\varphi(\cdot))}{\partial t^j}(t_2)\right| < \varepsilon\,\omega_j(s)\] for a.\,e. $s \in \mathbb{R}$.
\end{itemize}

\item [$(C_{2})$] It holds that $\omega_j\,\varphi,\,\frac{\partial ^j k}{\partial t^j}(t,\cdot)\,\eta(\cdot)\,\varphi(\cdot)\in\Lsp{1}(\bR)$   for every $t\in\bR$, $j=0,\dots,n$; and
\[\frac{\partial ^{j-l}}{\partial t^{j-l}}\frac{1}{\varphi}(t)\int_{-\infty}^{+\infty} \left|\frac{\partial ^l k}{\partial t^l}(t,s)\,\eta(s)\right| \, \varphi(s)\,\dif s\in \Lsp{\infty}(\bR),\]
for all $j=0,\dots,n$; $l=0,\dots,j$.

Moreover, defining
\[z_{(\pm)}(s):=\lim\limits_{t\rightarrow \pm \infty} \frac{|k(t,s)\,\eta(s)|}{\varphi(t)} \]
and
\[M(s):=\sup_{t\in\bR} \frac{|k(t,s)\,\eta(s)|}{\varphi(t)},\]
it is satisfied that $z_{(\pm)}\,\varphi, \, M\,\varphi \in L^1(\bR)$.
		
\item  [ $(C_{3})$] $f\colon \bR \times \bR \rightarrow [0,+\infty)$  satisfies a sort of $\Lsp{\infty}$-Carath\'{e}odory 
conditions, that is, $f(\cdot,y)$ is measurable for each fixed
$y\in\bR$ and $f(t,\cdot)$ is continuous for a.\,e. $t\in \bR$, and, for each $r>0$, there exists $\phi_{r} \in \Lsp{\infty}(\bR)$ such that \[\frac{f(t,x\varphi(t))}{\varphi(t)}\le\phi_{r}(t),\] for all $x\in [-r,r]$ and a.\,e. $t\in \bR$.
\item  [ $(C_4)$] $\alpha(|k(\cdot,s)\,\eta(s)|)\ge 0$ for a.\,e. $s\in\bR$.
\item  [ $(C_5)$] $\alpha(|k(\cdot,s)\,\eta(s)|)\,\varphi(s)\in \Lsp{1}(\bR)$ and \begin{displaymath}\alpha(L_1 u)\ge \int_{-\infty}^{\infty} \alpha(|k(\cdot,s)\,\eta(s)|)\,u(s)\, \dif s \ \text{for all } u\in P.\end{displaymath}
\item[$(C_6)$] There exists $A\subset \bR$ such that $A$ is a finite union of compact intervals and
$k(t,s)\, \eta(s) \ge 0, \ k(\cdot,s)\, \eta(s) \not\equiv 0$  for every $t\in A$ and a.\,e. $ s\in \bR$. Moreover, it holds that
\begin{equation*} 
	\frac{1}{\widetilde M(A)}=\frac{1}{\widetilde M} :=\inf_{t\in  A }\int_A k(t,s)\eta(s)\,\dif s>0.
\end{equation*}
\end{itemize}

We will also define the following auxiliary operator on $\widetilde\cC^n_\varphi$.
\[
L_2 \,u(t):=  \int_A \left(k(t,s) \, \eta(s)\right)^+\,u(s)\,\dif s.
\]

With regard to operator $L_2$, we will consider the following assumptions:
\begin{itemize}
\item  [ $(C_7)$] $\alpha((k(\cdot,s)\eta(s))^+)\ge 0$ for a. e. $s\in\bR$.
\item  [ $(C_8)$] $\alpha\left((k(\cdot,s)\,\eta(s))^+\right)\,\varphi(s)\in \Lsp{1}(A)$ and
\begin{displaymath}\alpha\left(L_2  u\right)\ge \int_A  \alpha\left((k(\cdot,s)\,\eta(s))^+\right)\,u(s)\, \dif s \ \text{for all } u\in P.\end{displaymath}
\end{itemize}

Finally, to ensure that operator $T$ maps the cone $K_\alpha$ into itself, we need to ask for the following conditions:
\begin{itemize}
	\item  [ $(C_9)$] $\alpha(k(\cdot,s)\eta(s))\ge 0$ for a. e. $s\in\bR$.
	\item  [ $(C_{10})$] $\alpha\left(k(\cdot,s)\,\eta(s)\right)\,\varphi(s)\in \Lsp{1}(\bR)$ for a.\,e. $s\in\bR$ and
	\begin{displaymath}\alpha\left(T  u\right)\ge \int_{-\infty}^{\infty} \alpha\left(k(\cdot,s)\,\eta(s)\right)\,f(s,u(s))\, \dif s \ \text{for all } u\in K_\alpha.\end{displaymath}
\end{itemize}

\begin{thm}\label{lcomp} 
	If $(C_1)$, $(C_2)$, $(C_4)$ and $(C_5)$ hold, then operator $L_1$ is continuous, compact and maps $P$ into $P\cap K_\alpha$.
\end{thm}

\begin{proof}
We will distinguish two different cases:

\textbf{CASE I: $n=0$:}

\emph{$L_1$ maps $(\widetilde\cC_\varphi,\|\cdot\|_\varphi)$ to $(\widetilde\cC_\varphi,\|\cdot\|_\varphi)$:} Let $u\in \widetilde\cC_\varphi$. From $(C_1)$, (i), given $\varepsilon\in\bR^+$, there exists some $\delta\in\bR^+$ such that for $t_1,\,t_2 \in \mathbb{R}$, $|t_1-t_2|<\delta$ it is satisfied that
\begin{equation}\label{cont_L_1 u_CASE1}\begin{split}
\left| \widetilde{L_1 u}(t_1)- \widetilde{L_1 u}(t_2)\right|  \le & \int_{-\infty}^{+\infty} \left| \frac{|k(t_1,s)\,\eta(s)|}{\varphi(t_1)}- \frac{|k(t_2,s)\,\eta(s)|}{\varphi(t_2)} \right| \, |u(s)| \, \dif s  \\ \le & \, \varepsilon \int_{-\infty}^{\infty} \omega_0(s)\,|u(s)| \, \dif s \le \varepsilon \int_{-\infty}^{\infty} \omega_0(s)\,\frac{|u(s)|}{\varphi(s)}\,\varphi(s) \, \dif s \\ \le & \, \varepsilon \, \|u\|_\varphi \int_{-\infty}^{\infty} \omega_0(s)\,\varphi(s)\, \dif s
\end{split}\end{equation}
and since, by $(C_2)$, $\omega_0\,\varphi\in \Lsp{1}(\bR)$, the previous expression is bounded from above by $\varepsilon\, \|u\|_\varphi\,c$ for some positive constant $c$. Hence, $ \widetilde{L_1 u}$ is continuous in $\bR$. Now we will prove that there exists
\[\lim\limits_{t\to \pm\infty}\widetilde{L_1 u}(t)=\lim\limits_{t\to \pm\infty}\frac{L_1 u(t)}{\varphi(t)}=\lim\limits_{t\to \pm\infty}\frac{1}{\varphi(t)} \int_{-\infty}^{\infty} |k(t,s)\,\eta(s)|\,u(s)\, \dif s \in \bR.\] 
Since $k(\cdot,s)\,\eta(s)\in\til{\cC}_\varphi$, then, for all $s\in\bR$, there exists
\[\lim\limits_{t\to \pm\infty}\frac{|k(t,s)\,\eta(s)|}{\varphi(t)}=:z_{(\pm)}(s)\in \bR. \]
On the other hand, for a.\,e. $s\in\bR$:
\[\left|\frac{|k(t,s)\,\eta(s)|}{\varphi(t)} \,u(s)\right|\le   M(s) \,|u(s)|= M(s) \, \frac{|u(s)|}{\varphi(s)} \varphi(s) \le \|u\|_\varphi\, M(s) \, \varphi(s)  \ \text{ for all } t\in\bR\]
and, from $(C_2)$, $M \, \varphi \in \Lsp{1}(\bR)$. Thus, from Lebesgue's Dominated Convergence Theorem,
\[\lim\limits_{t\to \pm\infty}\frac{1}{\varphi(t)} \int_{-\infty}^{\infty} |k(t,s)\,\eta(s)|\,u(s)\, \dif s= \int_{-\infty}^{\infty} \lim\limits_{t\to \pm\infty} \frac{|k(t,s)\,\eta(s)|}{\varphi(t)} \,u(s)\, \dif s= \int_{-\infty}^{\infty} z_{(\pm)}(s) \,u(s)\, \dif s, \]
and, since, 
\[\left|\int_{-\infty}^{\infty} z_{(\pm)}(s) \,u(s)\, \dif s\right| \le \int_{-\infty}^{\infty} z_{(\pm)}(s) \,|u(s)|\, \dif s \le \|u\|_\varphi \int_{-\infty}^{\infty} z_{(\pm)}(s) \,\varphi(s)\, \dif s\in\bR,\]
we deduce that $z_{(\pm)}\,u\in\Lsp{1}(\bR)$. Therefore there exists $\lim\limits_{t\to \pm\infty}\frac{L_1u(t)}{\varphi(t)}$. Consequently, $L_1u\in \widetilde\cC_\varphi$.

It is left to see that $L_1 u$ is bounded in $\|\cdot\|_\varphi$. Using $(C_2)$, we have that
\begin{equation}\label{bound_L_1 u_CASE1}\begin{split}
\|L_1 u\|_\varphi=& \left\|\widetilde{L_1 u}\right\|_\infty= \left\|  \frac{L_1 u}{\varphi}\right\|_\infty 
=  \left\|\frac{1}{\varphi(t)} \int_{-\infty}^{+\infty} |k(t,s)\,\eta(s)|\,u(s)\,\dif s \right\|_\infty \\ \le & \,\|u\|_\varphi\,\left\|\frac{1}{\varphi(t)}\int_{-\infty}^{+\infty} \left|k(t,s)\,\eta(s)\right|\, \varphi(s)\, \dif s\right\|_\infty < +\infty.
\end{split}\end{equation}

\emph{Continuity:} It is obvious from the linearity and boundedness of operator $L_1$.

\emph{Compactness:} Let $B \subset \widetilde{\mathcal{C}}_{\varphi}$ a bounded set, that is, $\|u\|_{\varphi} \le R$ for all $u \in B$ and some $R>0$. 
Then, in the upper bound of $\left\|\widetilde{L_1 u}\right\|_\infty$ found in expression \eqref{bound_L_1 u_CASE1} we can substitute $\|u\|_\varphi$ by $R$ and to obtain an upper bound which does not depend on $u$. Therefore it is clear that the set $L_1(B)$ is totally bounded.

On the other hand, taking into account the upper bound found in \eqref{cont_L_1 u_CASE1}, we have that if $t_1,\, t_2\in\bR$ are such that $|t_1-t_2|<\delta$ then
\begin{equation*}
\left| \widetilde{L_1 u}(t_1)- \widetilde{L_1 u}(t_2)\right| \le  \varepsilon \, \|u\|_\varphi \int_{-\infty}^{\infty} \omega_0(s)\,\varphi(s) \dif s \le \varepsilon \,R \int_{-\infty}^{\infty} \omega_0(s)\,\varphi(s) \dif s,
\end{equation*} 
and, since $\omega_0\,\varphi\in \Lsp{1}(\bR)$, we conclude that $L_1(B)$ is equicontinuous. 

In conclusion, we derive, by application of Ascoli-Arzela's Theorem, that $L_1(B)$ is relatively compact in $\widetilde{\mathcal{C}}_{\varphi}$ and therefore $L_1$ is a compact operator.

\emph{$L_1$ maps $P$ to $P\cap K_\alpha$:} Since $L_1$ has a positive integral kernel, it clearly maps $P$ into $P$. Finally, it maps $P$ into $P\cap K_\alpha$ as a direct consequence of hypothesis $(C_4)$ and $(C_5)$.

\textbf{CASE II: $n\neq 0$:}

\emph{$L_1$ maps $(\widetilde\cC^n_\varphi,\|\cdot\|_\varphi)$ to $(\widetilde\cC^n_\varphi,\|\cdot\|_\varphi)$:} Let $u\in \widetilde\cC^n_\varphi$. Since $\frac{k(\cdot,s)\,\eta(s)}{\varphi(\cdot)}$ is integrable for every $s\in\bR$, we can use Leibniz's Integral Rule for generalised functions (see \cite[p. 484]{Jones}) to get
\[\frac{\partial ^j \til{L_1 u}}{\partial t^j}(t)=\frac{\partial ^j (L_1 u/\varphi)}{\partial t^j}(t)=\int_{-\infty}^{+\infty} \frac{\partial ^j(k(\cdot,s)\,\eta(s)/\varphi(\cdot))}{\partial t^j}(t)\,u(s)\dif s.\]

On the other hand, from $(C_1)$, given $\varepsilon\in\bR^+$, there exists some $\delta\in\bR^+$ such that for $t_1,\,t_2 \in \mathbb{R}$, $|t_1-t_2|<\delta$ it is satisfied that
\begin{equation}\label{cont_L_1 u}\begin{split}
\left|\frac{\partial ^j \widetilde{L_1 u}}{\partial t^j}(t_1)-\frac{\partial ^j \widetilde{L_1 u}}{\partial t^j}(t_2)\right|  \le & \int_{-\infty}^{+\infty} \left| \frac{\partial^j (k(\cdot,s)\,\eta(s)/\varphi(\cdot))}{\partial t^j}(t_1)- \frac{\partial^j (k(\cdot,s)\,\eta(s)/\varphi(\cdot))} {\partial t^j}(t_2) \right| \, |u(s)| \, \dif s  \\ \le & \, \varepsilon \int_{-\infty}^{\infty} \omega_j(s)\,|u(s)| \, \dif s \le \varepsilon \, \|u\|_\varphi \int_{-\infty}^{\infty} \omega_j(s)\,\varphi(s)\, \dif s.
\end{split}\end{equation}
and, since $\omega_j\,\varphi\in \Lsp{1}(\bR)$, the previous expression is bounded from above by $\varepsilon\, \|u\|_\varphi\,c$ for some positive constant $c$. Hence, $\frac{\partial ^j \widetilde{L_1 u}}{\partial t^j}$ is continuous in $\bR$ for $j=0,\dots,n$, that is, $\widetilde{L_1 u}\in\cC^n(\bR,\bR)$. Analogously to Case I, it can be proved that there exists $\lim\limits_{t\to \pm\infty}\widetilde{L_1 u}(t)$ and, consequently, $L_1 u\in \widetilde\cC^n_\varphi$. 

It is left to see that $L_1 u$ is bounded in $\|\cdot\|_\varphi$. Using the General Leibniz's Rule (for differentiation), it is clear that
\[\frac{\partial ^j \widetilde{L_1 u}}{\partial t^j}=\frac{\partial ^j (L_1 u/\varphi)}{\partial t^j}=\sum_{l=0}^j {j \choose l}\frac{\partial ^l L_1 u}{\partial t^l}\frac{\partial ^{j-l}}{\partial t^{j-l}}\frac{1}{\varphi}.\]
Moreover, from Leibniz's Integral Rule for generalised functions again,
\[\frac{\partial ^l L_1 u}{\partial t^l}(t)=\int_{-\infty}^{\infty} \frac{\partial ^l k}{\partial t^l}(t,s)\,\eta(s)\,u(s)\dif s.\]
Thus,
\begin{equation*}\begin{split}
\left\|\frac{\partial ^j \widetilde{L_1 u}}{\partial t^j}\right\|_\infty=&\left\|\sum_{l=0}^j {j \choose l}\frac{\partial ^l L_1 u}{\partial t^l}\frac{\partial ^{j-l}}{\partial t^{j-l}}\frac{1}{\varphi}\right\|_\infty\le\sum_{l=0}^j {j \choose l}\left\|\frac{\partial ^l L_1 u}{\partial t^l}\frac{\partial ^{j-l}}{\partial t^{j-l}}\frac{1}{\varphi}\right\|_\infty\\
= & \sum_{l=0}^j {j \choose l} \left\|\frac{\partial ^{j-l}}{\partial t^{j-l}}\frac{1}{\varphi}(t) \int_{-\infty}^{+\infty} \frac{\partial ^l k}{\partial t^l}(t,s)\,\eta(s)\,u(s)\,\dif s \right\|_\infty.
\end{split}\end{equation*}
It is satisfied that
\begin{align*}
\left|\frac{\partial ^{j-l}}{\partial t^{j-l}}\frac{1}{\varphi}(t)\int_{-\infty}^{+\infty} \frac{\partial ^l k}{\partial t^l}(t,s)\,\eta(s)\,u(s)\,\dif s\right| \le & \, \left|\frac{\partial ^{j-l}}{\partial t^{j-l}}\frac{1}{\varphi}(t)\right|\int_{-\infty}^{+\infty} \left|\frac{\partial ^l k}{\partial t^l}(t,s)\,\eta(s)\right|\,|u(s)|\,\dif s \\
\le & \,\|u\|_\varphi\,\left|\frac{\partial ^{j-l}}{\partial t^{j-l}}\frac{1}{\varphi}(t)\right|\int_{-\infty}^{+\infty} \left|\frac{\partial ^l k}{\partial t^l}(t,s)\,\eta(s)\right|\, \varphi(s)\, \dif s,
\end{align*}
and so, from two previous inequalities and taking into account condition $(C_2)$, we deduce that
\begin{equation}\label{bound_L_1 u}
\left\|\frac{\partial ^j \widetilde{L_1 u}}{\partial t^j}\right\|_\infty \le \|u\|_\varphi \sum_{l=0}^j {j \choose l} \left\|\frac{\partial ^{j-l}}{\partial t^{j-l}}\frac{1}{\varphi}(t)\int_{-\infty}^{+\infty} \left|\frac{\partial ^l k}{\partial t^l}(t,s)\,\eta(s)\right| \, \varphi(s)\,\dif s \right\|_\infty < +\infty.
\end{equation}

Therefore, $\|L_1 u\|_\varphi<+\infty$. \par 
\emph{Continuity:} Again, it is obvious from the linearity and boundedness of operator $L_1$.

\emph{Compactness:} The proof is analogous to Case I but using equations \eqref{bound_L_1 u} and \eqref{cont_L_1 u} instead of \eqref{bound_L_1 u_CASE1} and \eqref{cont_L_1 u_CASE1}.

\emph{$L_1$ maps $P$ to $P\cap K_\alpha$:} The proof is the same than in Case I.  
\end{proof}

\begin{thm}
	If $(C_1)$, $(C_2)$, $(C_6)$, $(C_7)$ and $(C_8)$ hold, then operator $L_2$ is continuous, compact and maps $P$ into $P\cap K_\alpha$.
\end{thm}

\begin{proof}
We will distinguish two different cases:
	
\textbf{CASE I: $n=0$:}
	
\emph{$L_2$ maps $(\widetilde\cC_\varphi,\|\cdot\|_\varphi)$ to $(\widetilde\cC_\varphi,\|\cdot\|_\varphi)$:} Let $u\in \widetilde\cC_\varphi$. Since $k(\cdot,s)\,\eta(s)\in \widetilde\cC_\varphi$ for all $s\in\bR$, it is clear that $\left(\frac{k(\cdot,s)\,\eta(s)}{\varphi(\cdot)}\right)^+\equiv \frac{(k(\cdot,s)\,\eta(s))^+}{\varphi(\cdot)} \in \cC(\bR)$ for all $s\in \bR$. 

Analogously to the proof for $L_1$, from $(C_1)$, (ii), given $\varepsilon\in\bR^+$, there exists some $\delta\in\bR^+$ such that for $t_1,\,t_2 \in \mathbb{R}$, $|t_1-t_2|<\delta$ it is satisfied that
\begin{equation}\label{cont_L_2 u_CASE1}
	\left| \widetilde{L_2 u}(t_1)- \widetilde{L_2 u}(t_2)\right|  \le \varepsilon \, \|u\|_\varphi \int_{A}\omega_0(s)\,\varphi(s)\, \dif s
\end{equation}
and, since $\omega_0\,\varphi\in \Lsp{1}(\bR)$, it can be deduced that $ \widetilde{L_2 u}$ is continuous in $\bR$. 

It is left to see that there exists
		\[\lim\limits_{t\to \pm\infty}\widetilde{L_2 u}(t)=\lim\limits_{t\to \pm\infty}\frac{L_2 u(t)}{\varphi(t)}=\lim\limits_{t\to \pm\infty}\frac{1}{\varphi(t)} \int_A  (k(t,s)\,\eta(s))^+\,u(s)\, \dif s \in \bR.\] 
Reasoning as before, since $k(\cdot,s)\,\eta(s)\in\til{\cC}_\varphi$, then for all $s\in\bR$ it is ensured the existence of 
\[\lim\limits_{t\to \pm\infty}\frac{(k(t,s)\,\eta(s))^+}{\varphi(t)}\le \lim\limits_{t\to \pm\infty}\frac{|k(t,s)\,\eta(s)|}{\varphi(t)}= z_{(\pm)}(s)\in \bR. \]
		
On the other hand,
\[\left|\frac{(k(t,s)\,\eta(s))^+}{\varphi(t)} \,u(s)\right| \le \left|\frac{|k(t,s)\,\eta(s)|}{\varphi(t)} \,|u(s)|\right| \le   M(s) \,|u(s)|= M(s) \, \frac{|u(s)|}{\varphi(s)} \varphi(s) \le \|u\|_\varphi\, M(s) \, \varphi(s)\]
for all $t\in\bR$. From $(C_2)$, $M \, \varphi \in \Lsp{1}(\bR)$ and so  $M \, \varphi \in \Lsp{1}( A )$. Thus, from Lebesgue's Dominated Convergence Theorem,
	\[\lim\limits_{t\to \pm\infty}\frac{1}{\varphi(t)} \int_A  (k(t,s)\,\eta(s))^+\,u(s)\, \dif s= \int_A  \lim\limits_{t\to \pm\infty} \frac{(k(t,s)\,\eta(s))^+}{\varphi(t)} \,u(s)\, \dif s,\]
and since 
	\[\left|\int_A  \lim\limits_{t\to \pm\infty} \frac{(k(t,s)\,\eta(s))^+}{\varphi(t)} \,u(s)\, \dif s\right| \le \int_A  z_{(\pm)}(s) \,|u(s)|\, \dif s \le \|u\|_\varphi \int_A  z_{(\pm)}(s) \,\varphi(s)\, \dif s\in\bR,\]
it can be concluded that there exists $\lim\limits_{t\to \pm\infty}\frac{L_2 u(t)}{\varphi(t)}$ and consequently $L_2 u\in \widetilde\cC_\varphi$.

It is left to see that $L_2 u$ is bounded in $\|\cdot\|_\varphi$.
	\begin{equation*}\begin{split}
		\left\|\widetilde{L_2 u}\right\|_\infty=& \left\|  \frac{L_2 u}{\varphi}\right\|_\infty 
		=  \left\|\frac{1}{\varphi(t)} \int_A  (k(t,s)\,\eta(s))^+\,u(s)\,\dif s \right\|_\infty \\ \le & \,\|u\|_\varphi\,\left\|\frac{1}{\varphi(t)}\int_A  \left(k(t,s)\,\eta(s)\right)^+\, \varphi(s)\, \dif s \right\|_\infty \\
		\le & \, \|u\|_\varphi\, \left\| \frac{1}{\varphi(t)}\int_A  \left|k(t,s)\,\eta(s)\right|\, \varphi(s)\, \dif s \right\|_\infty \\ \le & \, \|u\|_\varphi\, \left\|\frac{1}{\varphi(t)}\int_{-\infty}^{\infty} \left|k(t,s)\,\eta(s)\right|\, \varphi(s)\, \dif s \right\|_\infty,
	\end{split}\end{equation*}
and so, from $(C_2)$, we deduce that
	\begin{equation}\label{bound_L_2 u_CASE1}
		\left\| \widetilde{L_2 u}\right\|_\infty \le \|u\|_\varphi  \left\|\frac{1}{\varphi(t)}\int_{-\infty}^{+\infty} \left|k(t,s)\,\eta(s)\right| \, \varphi(s)\,\dif s \right\|_\infty < +\infty.
	\end{equation}
			
Therefore, $\|L_2 u\|_\varphi<+\infty$. \par 

\emph{Continuity:} It is obvious from the linearity and boundedness of operator $L_2$.
				
\emph{Compactness:} The proof is analogous to the one for operator $L_1$ (Theorem \ref{lcomp}) by considering equations \eqref{bound_L_2 u_CASE1} and \eqref{cont_L_2 u_CASE1} instead of \eqref{bound_L_1 u_CASE1} and \eqref{cont_L_1 u_CASE1}, respectively.
				
\emph{$L_2$ maps $P$ to $P\cap K_\alpha$:} Since $L_2$ has a positive integral kernel, it clearly maps $P$ into $P$. Finally, it maps $P$ into $P\cap K_\alpha$ as a direct consequence of hypothesis $(C_7)$ and $(C_8)$.

\textbf{CASE II: $n\neq 0$:}
The proof is analogous to the one made for operator $L_1$, with some small changes in the line of those introduced in Case I.
\end{proof}
	
Analogously to the two previous theorems, it can be proved that 
\begin{thm}
If $(C_1)-(C_3)$, $(C_9)$ and $(C_{10})$ hold, the operator $T$ is continuous, compact and maps $K_\alpha$ into $K_\alpha$.
\end{thm}		
\begin{proof}
The proof, except for the continuity, is analogous to previous theorems but taking into account the fact that
\[f(s,u(s))=f\left(s,\frac{u(s)}{\varphi(s)}\varphi(s)\right)\le \phi_{\|u\|_\varphi}(s) \,\varphi(s) \le \left\|\phi_{\|u\|_\varphi}\right\|_\infty \,\varphi(s).  \]

\emph{Continuity:} Since $T$ is not a linear operator, continuity can not be deduced from boundedness, on the contrary to previous theorems. Therefore, we shall prove that operator $T$ is continuous in a different way:
 
\textbf{CASE I: $n= 0$:}

Let $\{u_n\}_{n \in \bN}$ be a sequence which converges to $u$ in $\widetilde{\mathcal{C}}_{\varphi}$. Then, there exists some $R\in\bR$ such that $\|u_n\|_\varphi\le R$ for all $n\in\bN$ and it holds that 
\begin{equation}\label{ec-bound_f_cont}
f(s,u_n(s))=f\left(s,\frac{u_n(s)}{\varphi(s)}\varphi(s)\right)\le \phi_{R}(s) \,\varphi(s) \le \left\|\phi_{R}\right\|_\infty \,\varphi(s). 
\end{equation}

Moreover, $\lim\limits_{n\rightarrow\infty}\|u_n-u\|_\varphi=0$ implies that $\lim\limits_{n\rightarrow\infty}\|\frac{u_n}{\varphi}-\frac{u}{\varphi}\|_\infty=0$, from where we deduce that $\frac{u_n(s)}{\varphi(s)} \to \frac{u(s)}{\varphi(s)}$ for a.\,e. $s \in \mathbb{R}$. Therefore,  $u_n(s) \to u(s)$ for a.\,e. $s \in \mathbb{R}$. 

Thus it is clear that
\begin{equation*}\begin{split}
\left| \widetilde{T u_n}(t)\right| &=\left| \frac{1}{\varphi(t)} \, \int_{-\infty}^{+\infty} k(t,s)\,\eta(s)\,f(s,u_n(s))\,\dif s \right| \le \frac{1}{\varphi(t)} \, \int_{-\infty}^{+\infty} |k(t,s)\,\eta(s)|\,f(s,u_n(s))\,\dif s \\
& \le \left\|\phi_{R}\right\|_\infty \, \frac{1}{\varphi(t)} \, \int_{-\infty}^{+\infty} |k(t,s)\,\eta(s)|\,\varphi(s)\,\dif s
\le \left\|\phi_{R}\right\|_\infty \left\|\frac{1}{\varphi(t)}\int_{-\infty}^{+\infty} \left|k(t,s)\,\eta(s)\right|\varphi(s)\,\dif s\right\|_\infty  
\end{split}\end{equation*}
for all $t\in\bR$ and we obtain, by application of Lebesgue's Dominated Convergence Theorem, that $T u_n \to T u$ in $\widetilde{\mathcal{C}}_{\varphi}$. Hence, operator $T$ is continuous.

\textbf{CASE II: $n\neq 0$:}

Let $\{u_n\}_{n \in \bN}$ be a sequence which converges to $u$ in $\widetilde{\mathcal{C}}_{\varphi}^n$. As in Case I, there exists some $R\in\bR$ such that $\|u_n\|_\varphi\le R$ for all $n\in\bN$ and \eqref{ec-bound_f_cont} holds.

Moreover, as in Case I, $u_n(s) \to u(s)$ for a.\,e. $s \in \mathbb{R}$. 

Thus,
\begin{equation*}\begin{split}
\left|\frac{\partial^j  \widetilde{L_1 u_n}}{\partial t^j}(t)\right| &  \le \left\|\phi_{R}\right\|_\infty \, \sum_{l=0}^j {j \choose l} \left\|\frac{\partial ^{j-l}}{\partial t^{j-l}}\frac{1}{\varphi}(t)\int_{-\infty}^{+\infty} \left|\frac{\partial ^l k}{\partial t^l}(t,s)\,\eta(s)\right|\varphi(s)\,\dif s\right\|_\infty
\end{split}\end{equation*}
for all $t\in\bR$ and then, by application of Lebesgue's Dominated Convergence Theorem, $T u_n \to T u$ in $\widetilde{\mathcal{C}}_{\varphi}^n$. Hence, operator $T$ is continuous.
\end{proof}

The following Theorem is analogous to \cite[Theorem 4.5]{gi-pp-ft} and is proven  using the facts that the considered operators leave $P$ invariant and $P$ is total, combined with Krein-Rutman Theorem.  
\begin{thm}\label{specrad} Assume conditions $(C_1)$, $(C_2)$ and $(C_4)-(C_8)$ hold. Then, it holds that $r(L_1)>0$ and it is an eigenvalue of $L_1$ with an eigenfunction in $P$. An analogous result holds for $L_2 $.
\end{thm}
\begin{proof}
We will prove the result for $L_1$. Consider $v\in P$ such that $v\equiv 1$ in $A$. Then, for $t\in A$,
\[L_1 v(t)=\int_{-\infty}^{\infty}|k(t,s)\,\eta(s)|\,v(s)\,\dif s \ge \int_{A}|k(t,s)\,\eta(s)|\,v(s)\,\dif s=\int_{A}k(t,s)\,\eta(s)\,\dif s \ge \frac{1}{\widetilde{M}}.\]

Then, there exists some compact set $B$, with $A\subset B$ such that when $t\in B$,
\[\int_{A}|k(t,s)\,\eta(s)|\,\dif s \ge \frac{1}{2\,\widetilde{M}}.\]

Now, defining $u(t)=1$ for $t\in A$ and $u(t)=0$ when $t\notin B$, from Whitney's Extension Theorem \cite[Theorem I]{whitney}, $u$ can be extended to $\bR$ (and this extension will be also denoted by $u$) as a function of class $n$. Moreover, from the proof of Whitney's Extension Theorem, it is possible to deduce that this extension will be upperly bounded by $1$.

Finally, since $\lim\limits_{t\to \pm \infty} u(t)=0$, it is clear that $u\in \widetilde{\cC}_\varphi^n$ with independence of the choice of $\varphi$.

Therefore, for $t\in B$, it holds that 
\begin{equation*}\begin{split}
L_1 u(t)&=\int_{-\infty}^{\infty}|k(t,s)\,\eta(s)|\,u(s)\,\dif s \ge \int_{A}|k(t,s)\,\eta(s)|\,u(s)\,\dif s=\int_{A}|k(t,s)\,\eta(s)|\,\dif s \ge \frac{1}{2\,\widetilde{M}} \\ &\ge \frac{1}{2\,\widetilde{M}} \, u(t), 
\end{split}\end{equation*}
and for $t\notin B$,
\[L_1 u(t)=0 = \frac{1}{2\,\widetilde{M}}\,u(t). \]

Thus, as a consequence of Theorem \ref{lowbound}, we conclude that $r(L_1)\ge \frac{1}{2\,\widetilde{M}}>0$.

Finally, since $P$ is a total cone and $L_1$ maps $P$ into $P$, Krein-Rutman Theorem assures that $r(L_1)$ is an eigenvalue with an eigenvector $\phi \in P\setminus\{0\}$.
\end{proof}

\begin{rem}
	As a consequence of Theorems \ref{lcomp} and \ref{specrad}, we know that the eigenfunctions mentioned above are in $P\cap K_\alpha$.
\end{rem}

We will define the following operator on $\cC^n(A,\bR)$
	\[\bar{L}u(t):=\int_A  k(t,s)\,\eta(s)\,u(s)\, \dif s, \quad t\in A, \]
	and consider the cone  $P_{ A }$ of positive functions in $\cC^n(A,\bR)$.
	
As with previous operators, we will prove that
\begin{thm}\label{barL}
Assume that conditions $(C_1)$, $(C_2)$ and $(C_6)-(C_8)$ hold.	Then, operator $\bar{L}$ is compact and maps $P_A$ into $P_A$.
\end{thm}
	
\begin{proof}
Let $f\in\cC^n(A,\bR)$ and $B\subset\bR$ an open and bounded set such that $A\subset B$. Define now $g(t)=f(t)$ for $t\in A$ and $g(t)=0$ for $t\in \bR\setminus B$. Then, from Whitney's Extension Theorem \cite[Theorem I]{whitney}, $g$ can be extended to $\bR$ as a function of class $n$, that is, there exists an extension of $f$ to $\bR$ as a function of class $n$ such that this extension vanishes for $t\in \bR\setminus B$. Obviously, this extension of $f$ belongs to $\cC^n_\varphi(\bR)$.

Now, denote by $i$ the function which maps a function in $\cC^n(A,\bR)$ to the aforementioned extension in $\cC^n_\varphi(\bR)$ and by $\pi$ the map which takes every function in $\cC^n_\varphi(\bR)$ to its restriction to the set $A$ (which clearly belongs to $\cC^n(A,\bR)$). We obtain the following diagram:
	
\begin{center}
	\tikzset{node distance=3cm, auto}
	\begin{tikzpicture}
	\node (1) {$\cC^n_\varphi(\bR)$};
	\node (2) [right of=1] {$\cC^n_\varphi(\bR)$};
	\node (3) [below of=1, yshift=1.5cm] {$\cC^n(A,\bR)$};
	\node (4) [right of=3] {$\cC^n(A,\bR)$};
	\draw[->] (1) to node {$L_2$} (2);
	\draw[->] (3) to node {$\bar{L}$} (4);
	\draw[->] (3) to node {$i$} (1);
	\draw[->] (2) to node {$\pi$} (4);
	\end{tikzpicture}
\end{center}

Let us show now that it is commutative. Consider $f\in \cC^n(A,\bR)$. It holds that
\begin{equation*}\begin{split}
(\pi\circ L_2 \circ i) (f)(t) & =\pi\left(\int_{A}\left(k(t,s)\,\eta(s)\right)^+\,  i(f)(s)\,\dif s \right) = \pi\left(\int_{A}\left(k(t,s)\,\eta(s)\right)^+\, f(s)\,\dif s \right) \\
&=\int_{A}k(t,s)\,\eta(s)\, f(s)\,\dif s=\bar{L}(f)(t), \quad t\in A.
\end{split}\end{equation*}

Now, since $L_2$ is compact and both $i$ and $\pi$ are continuous, we deduce that $\bar{L}$ is a compact operator.

Finally, from $(C_6)$ it is clear that $\bar{L}$ maps $P_A$ into $P_A$.
\end{proof}

\begin{rem}
We point out that, in the previous proof, Whitney's extension theorem can be used as a consequence of the fact that $A$ is a finite union of compact intervals.
\end{rem}

\begin{thm}It holds that $r(\bar{L})>0$ and it is an eigenvalue of $\bar{L}$ with an eigenfunction in $P_A$.
\end{thm}

\begin{proof}
Let $\psi$ be the eigenfunction related to $L_2$ whose existence is proved in Theorem \ref{specrad}. Then, if we consider its restriction to $A$, $\psi|_{A}$, it is clear that for $t\in A$
\[\bar{L}\psi|_{A}(t)=L_2\psi(t)=r(L_2)\,\psi(t)=r(L_2)\,\psi|_{A}(t), \]
and so from Theorems \ref{lowbound} and \ref{specrad}, we deduce that $r(\bar{L})\ge r(L_2)>0$.
\end{proof}

We define the following numbers in the extended real line:
\begin{equation*}
\begin{split}
f^{0}=\varlimsup_{x\to 0}\frac{\sup\limits_{t \in \bR} \, \dfrac{f(t,x\varphi(t))}{\varphi(t)}}{|x|},& \quad
f_{0}=\varliminf_{x\to 0}\frac{\inf\limits_{t \in A} \,\dfrac{f(t,x\varphi(t))}{\varphi(t)}}{|x|}, \\[.2pt]
f^{\infty}=\varlimsup_{|x| \to+\infty}\frac{\sup\limits_{t \in \bR} \, \dfrac{f(t,x\varphi(t))}{\varphi(t)}}{|x|},&
\quad f_{\infty}=\varliminf_{|x| \to+\infty} \frac{\inf\limits_{t \in A} \, \dfrac{f(t,x\varphi(t))}{\varphi(t)}}{|x|}.
\end{split}
\end{equation*}

To prove that the index of some subsets of a cone is $1$ or $0$, we will use the following well-known sufficient conditions.

Let $K$ be a cone in a Banach space $X$. If $\Omega\subset X$ is an open and bounded subset of $K$ (in the relative topology), we denote by $\overline{\Omega}$ and $\partial \Omega$, respectively,
its closure and its boundary relative to $K$. Moreover, we will note $\Omega_K=\Omega \cap K$, which is an open subset of $K$ in the relative topology.

\begin{lem}\label{lemind}
	Let $\Omega$ be an open bounded set in $K$ with $0\in \Omega_{K}$ and $\overline{\Omega_{K}}\ne K$. Assume that $F:\overline{\Omega_{K}}\to K$ is
	a continuous compact map such that $x\neq Fx$ for all $x\in \partial \Omega_{K}$. Then the fixed point index\index{fixed point index} $i_{K}(F, \Omega_{K})$ has the following properties.
	\begin{itemize}
		\item[(1)] If there exists $e\in K\backslash \{0\}$ such that $x\neq Fx+\lambda e$ for all $x\in \partial \Omega_K$ and all $\lambda
		\ge0$, then $i_{K}(F, \Omega_{K})=0$.
		\item[(2)] If  $\lambda x \neq Fx$ for all $x\in \partial \Omega_K$ and for every $\lambda \geq 1$, then $i_{K}(F, \Omega_{K})=1$.
		\item[(3)] If $i_K(F,\Omega_K)\ne0$, then $F$ has a fixed point in $\Omega_K$.
		\item[(4)] Let $\Omega^{1}$ be open in $X$ with $\overline{\Omega^{1}}\subset \Omega_K$. If $i_{K}(F, \Omega_{K})=1$ and
		$i_{K}(F, \Omega_{K}^{1})=0$, then $F$ has a fixed point in $\Omega_{K}\backslash \overline{\Omega_{K}^{1}}$. The same result holds if
		$i_{K}(F, \Omega_{K})=0$ and $i_{K}(F, \Omega_{K}^{1})=1$.
	\end{itemize}
\end{lem}

\begin{dfn} Let, $X,Y,Z$ be topological spaces, $Y$ Hausdorff. Let $f:X\to Y$, $g:X\to Z$. Let $z_0\in g(X)'$. We say that $L$ is the limit of $f$ when $g(x)$ tends to $z_0$ if for every neighborhood $N_Y$ of $L$ there exists a neighborhood $N_Z$ of $z_0$ such that $f\(g^{-1}\(N_Z\backslash\{z_0\}\)\)\subset N_Y$. We write
\[\lim_{g(x)\to z_0}f(x)=L.\]
\end{dfn}

A particular case of this definition would be the notion of limit in the case of the topology occurring when studying Stieltjes derivatives with respect to a function $g$ (cf. \cite{PoRo,FP2016}).

In order to prove the following Theorem, we adapt some of the proofs of \cite[Theorems 3.2-3.5]{jwkleig} to this new context.

\begin{thm}\label{thmindeig} Assume that $(C_1)-(C_{10})$ hold. 
Assume also that there exists $\beta:\cC^n_\varphi \rightarrow [0,\infty)$ such that
	\[\lim_{\beta(u)\to0}\|u\|_\varphi=0, \quad \lim_{\beta(u)\to +\infty}\|u\|_\varphi=+\infty, \]
	and
	\[\beta(u)\neq 0 \Rightarrow u\not\equiv 0. \]

Consider $K_\alpha^{\beta,\,\rho}:=\left\{u\in K_\alpha\ :\  \beta(u)<\rho \right\}$. We have the following.
	\begin{enumerate}
		\item[$(1)$] If $\;0\le f^{0}<\mu(L_1)$, then there exists $\rho_{0}>0$ such
		that
		$
		i_{K_\alpha}(T,K_\alpha^{\beta,\,\rho})=1$ for each $\rho\in (0,\rho_{0}].$
		
		\item[$(2)$] If $\;0\le f^{\infty}<\mu(L_1)$, then there exists $R_{0}>0$ such
		that
		$
		i_{K_\alpha}(T,K_\alpha^{\beta,\,R})=1$ for each $R > R_{0}.
		$
		\item[$(3)$] If $\mu\left(L_2 \right)<f_{0}\leq \infty$, then there exists
		$\rho_{0}>0$ such that 
		$
		i_{K_\alpha}(T,K_\alpha^{\beta,\,\rho})=0
		$
		for each 
		$\rho\in (0,\rho_{0}].$
		\item[$(4)$] If $\mu\left(L_2 \right)<f_{\infty} \leq \infty$, then there
		exists $R_{1}>0$ such that 
		$
		i_{K_\alpha}(T,K_\alpha^{\beta,\,R})=0$
		for each $R \geq R_{1}.$
	\end{enumerate}
\end{thm}

\begin{proof}		 
    $(1)$ Let $\tau>0$ be such that $f^0< \mu(L_1)-\tau=:\xi$. Then there exists $\til \rho_0\in(0,1)$ such that, for all $x\in[-\til \rho_0,\til \rho_0]$  and almost every $t\in \bR$, we have
	\begin{equation*}
	f(t,x\varphi(t))\le\xi|x|\varphi(t).
	\end{equation*}
	Also, since $\lim_{\beta(u)\to0}\|u\|_\varphi=0$, there is $\rho_0<\til \rho_0$ such that $\|u\|_\varphi<\til \rho_0$ for $u\in \ol{K_\alpha^{\beta,\,\rho_0}}$. Let $\rho\in (0,\rho_0]$.  We prove that $Tu\ne\lambda u$ for $u\in\partial K_\alpha^{\beta,\,\rho}$ and $\lambda\ge 1$, which implies $ i_{K}(T,K_\alpha^{\beta,\,\rho})=1$. In fact, if we assume otherwise, then there exists $u\in\partial K_\alpha^{\beta,\,\rho}$, (that is, $\beta(u)=\rho$ and therefore, $u\not\equiv 0$) and $\lambda\ge1$ such that $\l u=Tu$. Therefore, for $ t\in\bR$,
	\begin{align*}
	|u(t)|\leq &\,\lambda |u(t)|=  |Tu(t)|  =  \left|\int_{-\infty}^\infty k(t,s)\,\eta(s)\,f(s,u(s))\dif s\right|\\
	\le & \int_{-\infty}^\infty|k(t,s)\,\eta(s)|\,f\(s,\frac{u(s)}{\varphi(s)}\varphi(s)\)\dif s \le  \xi\int_{-\infty}^\infty|k(t,s)\,\eta(s)||u(s)|\dif s\\ 
	=&\,\xi\left(L_1|u|\right)(t).
	\end{align*}
	We conclude that $|u|\le\xi L_1|u|$. 	Thus, iterating, we have that
	\begin{align*}
	|u|\le & \xi \,L_1|u|\le\xi\, L_1\(\xi\, L_1|u|\)=\xi^2\,L_1^2|u|\le \dots\le \(\xi\, L_1\)^n|u|.
	\end{align*}
	That is,
	\[\|u\|_\varphi\le \xi^n\,\|L_1^n|u|\,\|_\varphi\]
and, hence,
\[1\le\xi^n\,\frac{\|L_1^n|u|\,\|_\varphi}{\|u\|_\varphi}\le \xi^n\,\| L_1^n\|_\varphi.\]
	Taking the $n$-th square root and the limit when $n\to\infty$,
	\[1\le \xi \, \left(\|L_1^n\|_\varphi\right)^\frac{1}{n}\to \xi\,r(L_1),\]
	a contradiction. \par

	$(2)$ Let $\tau\in\bR^+$ such that $f^\infty<\mu(L_1)-\tau=:\xi$. Then there exists $R_1>0$ such that for every $|x|\ge R_1$ and almost every $t\in \bR$
	\begin{displaymath}
	f(t,x\varphi(t))\le \xi\,|x|\varphi(t) .
	\end{displaymath}
	Also, by $(C_3)$ there exists $\phi_{R_1}\in L^\infty(\bR)$ such that \[\frac{f(t,x\varphi(t))}{\varphi(t)}\le\phi_{R_1}(t),\] for all $x\in[-R_1,R_1]$ and almost every $t\in \bR$. Hence,
	\begin{equation}\label{supest}
	f(t,x\varphi(t))\le\xi|x|\varphi(t) +\varphi(t)\phi_{R_1}(t)\ \text{for all}\  x\in \bR\ \text{and almost every }\ t\in \bR.
	\end{equation}
	Denote by $\Id$  the identity operator and observe that $\Id-\xi L_1$  is invertible since $\xi L_1$ has spectral radius less than one. Furthermore, by the Neumann series expression, 
	\begin{displaymath}
(\Id-\xi L_1)^{-1}=\sum_{k=0}^\infty(\xi L_1)^k
	\end{displaymath}
	and therefore,  $(\Id-\xi L_1)^{-1}$ maps $P$ into $P\cap K_\alpha$, since $L_1$ does.
	
Since $\phi_{R_1}\in L^\infty(\bR)$, 
	\[C(t):=\int_{-\infty}^{\infty} |k(t,s)\,\eta(s)|\,\varphi(s) \,\phi_{R_1}(s) \dif s\le \left\|\phi_{R_1}\right\|_{\infty} \int_{-\infty}^{\infty} |k(t,s)\,\eta(s)|\,\varphi(s) \dif s,\]
	and so, from $(C_2)$, it is clear that $C\in \widetilde \cC^n_\varphi$. Furthermore, since $C(t)\ge 0$ for all $t\in\bR$, $C\in P$. Therefore $(\Id-\xi L_1)^{-1} C \in P\cap K_\alpha$ and $R_0:= \| (\Id-\xi L_1)^{-1}C\|_\varphi<+\infty$.
	
	Because $\lim_{\beta(u)\to +\infty}\|u\|_\varphi=+\infty$, there exists $R_2>R_0$ such that $\|u\|_\varphi>R_1$ for every $u\in\partial K_\alpha^{\beta,\,R}$ for $R>R_2$. Now we prove that for each $R>R_2$, $T u\ne\lambda u$ for all $u\in\partial K_\alpha^{\beta,\,R}$ and $\lambda\ge 1$, which implies $ i_{K}(T,K_\alpha^{\beta,\,R})=1$. Assume otherwise: there exists $u\in\partial K_\alpha^{\beta,\,R}$ and $\lambda\ge 1$ such that $\lambda u=Tu$.  Taking into account the inequality \eqref{supest}, we have, for $t\in \bR$,
	\begin{align*}
	|u(t)|&\leq\lambda |u(t)|=  |Tu(t)|  =  \left|\int_{-\infty}^{\infty}k(t,s)\,\eta(s)\,f(s,u(s)) \, \dif s \right|\\ & \le  \int_{-\infty}^{\infty}|k(t,s)\,\eta(s)|\,f\(s,\frac{u(s)}{\varphi(s)}\varphi(s)\)\, \dif s  
	\le  \int_{-\infty}^{\infty}|k(t,s)\,\eta(s)|\left[\xi\left|\frac{u(s)}{\varphi(s)}\right|\varphi(s) +\varphi(s)\phi_{R_1}(s)\right]\, \dif s
	\\ & \le  \xi\int_{-\infty}^{\infty}|k(t,s)\,\eta(s)|\,|u(s)|\, \dif s+C(t)=\xi\,L_1 |u|(t)+C(t),
	\end{align*}
	which implies 
	\begin{displaymath}
	(\Id-\xi L_1)|u|(t)\le C(t).
	\end{displaymath}
	Since $(\Id-\xi L_1)^{-1}$ is non-negative, we have
	\begin{displaymath}
	|u(t)|\le (\Id-\xi L_1)^{-1} C(t)
	\end{displaymath}
	and, consequently,
	\[\|u\|_\varphi \le \|(\Id-\xi L_1)^{-1} C\|_\varphi =R_0. \]
	Therefore, we have $\|u\|_\varphi\le R_0<R$, a contradiction.\par
	$(3)$ There exists $\rho_0>0$ such that  for all $x\in(0,\rho_0]$ and  all $t\in A$ we have
	\begin{displaymath}
	f(t,x\varphi(t))\geq \mu\left(L_2 \right)\,x\varphi(t).
	\end{displaymath}
	Since  $\lim_{\beta(u)\to 0}\|u\|_\varphi=0$, there exists $\rho_1\in(0,\rho_0]$ such that $\|u\|_\varphi <\rho_0$ for every $u\in K_\a^{\beta,\,\rho}$, $\rho\in(0,\rho_1]$.  Let $\rho\in(0,\rho_1]$ be fixed. Let us prove that
	$u\ne Tu+\lambda\varphi_1$ for all $u$ in $\partial K_\a^{\beta,\,\rho}$ and $\lambda\geq 0$,
	where $\varphi_1\in K_\a\cap P$ is the eigenfunction of $L_2 $ with $\|\varphi_1\|=1$ corresponding to the eigenvalue $1/\mu\left(L_2 \right)$ of which the existence is proved in Theorem \ref{specrad}. This implies that $ i_{K}(T,K_{\rho})=0$.\par
	Assume, on the contrary, that there exist $u\in\partial K_\a^{\beta,\,\rho}$ and $\lambda\geq 0$ such that $u=Tu+\lambda\varphi_1$. We distinguish two cases. Firstly, we discuss the case $\lambda>0$. We have, for $t\in A$ in the conditions of $(C_6)$,
	\begin{align*}
	u(t)=&\, \int_{-\infty}^{\infty} k(t,s)\,\eta(s)f(s,u(s))\dif s +\lambda \, \varphi_1(t) \ge \int_A k(t,s)\,\eta(s) \,f\(s,\frac{u(s)}{\varphi(s)}\varphi(s)\)\dif s +\lambda \, \varphi_1(t)
	\\ \geq & \,\mu\left(L_2 \right) \int_A k(t,s)\,\eta(s)\,u(s)\dif s+\lambda \varphi_1(t) =\mu\left(L_2 \right) L_2 u(t)+\lambda \varphi_1(t)\ge \lambda\varphi_1(t).
	\end{align*}
	
	Hence, \[(L_2 u)(t)\ge\lambda (L_2 \varphi_1)(t)= \dfrac{\lambda}{\mu\left(L_2 \right)}\varphi_1(t),\] in such a way that we obtain 
	\begin{displaymath}
	u(t)\ge\mu\left(L_2 \right) L_2 u(t)+\lambda\varphi_1(t)\ge2\lambda\varphi_1(t),\ \text{   for   } t\in A.
	\end{displaymath}
	
	By iteration, we deduce that, for $ t\in A$, we get
	\begin{displaymath}
	u(t)\ge n\lambda\varphi_1(t)  \text{   for every  } n\in\bN,
	\end{displaymath}
	a contradiction because $u(t)$ is finite and  $\varphi_1|_A\not\equiv0$.\par

	Now we consider the case $\lambda=0$.  Let $\varepsilon>0$ be such that for all such that  for all $x\in(0,\rho_0]$ and  and almost every $t\in A$ we have
	\begin{equation*}
	f(t,x\varphi(t))\geq (\mu\left(L_2 \right)+\varepsilon)\, x\varphi(t).
	\end{equation*}
	We have, for $t\in A$, 
	\begin{equation*}
	u(t)=\int_{-\infty}^{\infty} k(t,s)\,\eta(s)\,f(s,u(s)) \dif s \geq\int_A \left(k(t,s)\,\eta(s)\right)^+\,f(s,u(s)) \dif s \geq(\mu\left(L_2 \right)+\varepsilon)\, L_2 u(t).
	\end{equation*}
	
	Since $L_2 \varphi_1(t)=r\left(L_2 \right)\varphi_1(t)$ for $t\in\bR$, we have, for $t\in A$,
	\begin{displaymath}
	\bar{L} \varphi_1(t)=L_2 \varphi_1(t)=r\left(L_2 \right)\varphi_1(t),
	\end{displaymath}
	and we obtain $r(\bar{L})\geq r\left(L_2 \right)$. On the other hand, we have, for $t\in A$, 
	\begin{equation*}
	u(t)\geq(\mu\left(L_2 \right)+\varepsilon)\, L_2 u(t)=(\mu\left(L_2 \right)+\varepsilon) \, \bar{L} u(t),
	\end{equation*}
	where $u(t)>0$. Thus, using Theorem~\ref{upbound}, we have  $r(\bar{L})\leq \dfrac{1}{\mu\left(L_2 \right)+\varepsilon}$
	and therefore $r\left(L_2 \right)\leq \dfrac{1}{\mu\left(L_2 \right)+\varepsilon}$. This gives $\mu\left(L_2 \right)+\varepsilon\leq \mu\left(L_2 \right)$, a contradiction. \par
	
	$(4)$ Let $R_1>0$ be such that 
	\begin{displaymath}
	f(t,x\varphi(t))>\mu\left(L_2 \right) \, x\varphi(t)
	\end{displaymath}
	for all $x>R_1$ and  all $t\in A$. 
	
	Moreover, since $\lim_{\beta(u)\to +\infty}\|u\|_\varphi=+\infty$, there exists $R_2$ such that $\|u\|>R_1$ for every $u\in\partial K_\alpha^{\beta,\,R}$ for $R>R_2$.
	
	Let  $R\geq R_2$. Now, proceeding as in the proof of the statement~$(3)$, it is easy to prove that $u\ne Tu+\lambda\varphi_1$ for all $u$ in $\partial K_\a^{\beta,\,R}$ and $\lambda\geq 0$, which  implies $ i_{K}(T,K_\a^{\beta,\,R})=0$. 
\end{proof}

The following Theorem, in the line of \cite{jwmz-na}, applies the index results in Theorem \ref{thmindeig}  in order to get some results on existence of nontrivial solutions for the equation \eqref{eqthamm}.

\begin{thm} \label{thmones}
	Assume that conditions $(C_1)-(C_{10})$ hold. Suppose also that one of the following conditions is satisfied 
	\begin{enumerate}
		\item[$(T_{1})$] $0\le f^{0}<\mu(L_1)$ and $\mu\left(L_2 \right)<f_{\infty}\le \infty$.
		\item[$(T_{2})$] $0\le f^{\infty}<\mu(L_1)$ and $\mu\left(L_2 \right)<f_{0}\le \infty$.
	\end{enumerate}
	Then the integral equation~\eqref{eqthamm} has at least one non-trivial solution in $K_\alpha$.
\end{thm}
\begin{proof}
	We will prove $(T_1)$, being $(T_2)$ analogous.
	
	Take $\beta(u)=\|u\|_\varphi$. Clearly $\beta$ is in the conditions of Theorem \ref{thmindeig}. Then, the existence of $\rho_0>0$ and $R_1>0$ such that $i_{K_\alpha}(T, K_\alpha^{\beta,\rho})=1$ for each $\rho\in(0,\rho_0]$ and $i_{K_\alpha}(T, K_\alpha^{\beta,R})=0$ for each $R\ge R_1$ is ensured. 
	
	Therefore, if we choose $\rho\le \rho_0$ and $R\ge R_1$ such that $\rho<R$, $K_\alpha^{\beta,\rho} \subset K_\alpha^{\beta,R}$ and from (4) in Lemma \ref{lemind} we deduce that $T$ has a fixed point in $K_\alpha^{\beta,R}\setminus \overline{K_\alpha^{\beta,\rho}}$.
\end{proof}

The following Lemma establishes some relations between the characteristic values of some of the considered operators.
\begin{lem} It holds that
$\widetilde M(A)\geq \mu\left(L_2 \right)\geq\mu( L_1)$.\end{lem}
\begin{proof}
	First, we prove that $\mu\left(L_2 \right)\geq\mu( L_1)$. Let $\phi$ be an eigenfunction of $L_1$ related to the eigenvalue $r(L_1)$. We have that
	\begin{equation*}\begin{split}
	r(L_1)\,\phi(t)&=L_1\phi(t)= \int_{-\infty}^{\infty}\left|k(t,s)\,\eta(s) \right|\phi(s) \dif s \ge \int_A \left|k(t,s)\,\eta(s) \right|\phi(s) \dif s \\& \ge \int_A \left(k(t,s)\,\eta(s) \right)^+\phi(s) \dif s= L_2  \phi(t). 
	\end{split}	\end{equation*}
	Therefore, Theorem \ref{upbound} yields that $r\left(L_2 \right)\le r(L_1)$ or, equivalently, $\mu\left(L_2 \right)\geq\mu( L_1)$.
	
	Now we prove $\widetilde M(A)\geq \mu\left(L_2  \right)$.
	Let $\phi \in P\cap K_\alpha$ be a corresponding eigenfunction of norm $1$ of $1/\mu \left(L_2 \right)$ for the operator $L_2 $, that is $\phi=\mu\left(L_2 \right)L_2  (\phi)$ and $\Vert \phi\Vert=1$.
	Then, for $t\in  A $, we have
	\begin{equation*}
	\phi(t)= \mu\left(L_2 \right)\int_a^b k(t,s)\,\eta(s)\,\phi(s) \dif s\geq\mu\left(L_2 \right)\min_{t\in  A }\phi(t)\int_a^b k(t,s)\,\eta(s) \dif s
	.\end{equation*}
	Taking the infimum over $ A $, we obtain 
	\begin{equation*}
	\min_{t\in  A }\phi(t)\geq  \mu\left(L_2  \right) \min_{t\in  A }\phi(t)/ \widetilde M(A),
	\end{equation*}
	that is, $\widetilde M(A)\geq \mu\left(L_2  \right)$.
\end{proof}

\begin{rem}
		We note that the previous results could also be formulated for $\til\cC_\varphi([a,+\infty))$ or $\til\cC_\varphi((-\infty,a])$ (with obvious notation) for any $a\in\bR$ (see \cite{Somoza}).
\end{rem}

\section{An example}
We will consider now the problem
\begin{equation*}
Tu(t)=\int_{-\infty}^{\infty} e^{-\frac{|s|}{2}}\,\sin t\, \sqrt{|u(s)|} \, \sin^2 s\, \dif s,
\end{equation*}
that is, $k(t,s)=e^{-\frac{|s|}{2}}\,\sin t$, $\eta(s)=1$ and $f(s,y)=\sqrt{|y|} \, \sin^2 s$.

We will take 
\[\varphi(t)=|t|,\]
and
\[\alpha(u)=\min_{t\in\left[\frac{\pi}{4},\frac{3\,\pi}{4}\right]} u(t)-\frac{\sqrt{2}}{2}\,\|u\|_\infty.\]

We will verify that conditions $(C_1)-(C_8)$ are satisfied for the case $n=0$:

\begin{itemize}
\item[$(C_1)$] First of all, since $k(\cdot,s)\in\cC(\bR)$ and there exist
\[\lim\limits_{t\to\pm\infty} \frac{k(t,s)}{\varphi(t)} = \lim\limits_{t\to\pm\infty} \frac{e^{-\frac{|s|}{2}}\,\sin t}{|t|}=0, \]
it is clear that $k(\cdot,s)\in\widetilde{\cC}_\varphi$ for all $s\in\bR$.

Moreover, for every $\varepsilon>0$ there exists $\delta>0$ such that when $|t_1-t_2|<\delta$,
\begin{itemize}
	\item[(i)] \[\left|\frac{k(t_1,s)}{\varphi(t_1)}-\frac{k(t_2,s)}{\varphi(t_2)} \right|= \left|\frac{e^{-\frac{|s|}{2}}\, \sin t_1}{|t_1|}-\frac{e^{-\frac{|s|}{2}}\, \sin t_2}{|t_2|} \right| \le \varepsilon\,e^{-\frac{|s|}{2}}, \]
\end{itemize}
\begin{itemize}
	\item[(ii)]
\[\left|\frac{|k(t_1,s)|}{\varphi(t_1)}-\frac{|k(t_2,s)|}{\varphi(t_2)} \right|= \left|\frac{e^{-\frac{|s|}{2}}\, |\sin t_1|}{|t_1|}-\frac{e^{-\frac{|s|}{2}}\, |\sin t_2|}{|t_2|} \right| \le \varepsilon\,e^{-\frac{|s|}{2}}, \]
\end{itemize}
and
\begin{itemize}
	\item[(iii)]
\[\left|\frac{(k(t_1,s)^+}{\varphi(t_1)}-\frac{(k(t_2,s))^+}{\varphi(t_2)} \right|= \left|\frac{e^{-\frac{|s|}{2}}\, (\sin t_1)^+}{|t_1|}-\frac{
e^{-\frac{|s|}{2}}\, (\sin t_2)^+}{|t_2|} \right| \le \varepsilon\,e^{-\frac{|s|}{2}},\]
\end{itemize}
so we will take $\omega_0(s)=e^{-\frac{|s|}{2}}$.
\item[$(C_2)$] Clearly, it holds that $\omega_0\,\varphi \in \Lsp{1}(\bR)$. Also,
\[\frac{1}{\varphi(t)}\int_{-\infty}^{\infty} |k(t,s)|\,\varphi(s)\,\dif s = \frac{|\sin t|}{|t|}\int_{-\infty}^{\infty} e^{-\frac{|s|}{2}}\,|s|\,\dif s =8\, \frac{|\sin t|}{|t|}  \in \Lsp{\infty}(\bR).\]

Moreover, in this case
\[z_{(\pm)}(s)=\lim\limits_{t\to\pm\infty}\frac{e^{-\frac{|s|}{2}}\,|\sin t|}{|t|}=0,\]
\[M(s)=\sup_{t\in\bR} \,\frac{e^{-\frac{|s|}{2}}\,|\sin t|}{|t|} =e^{-\frac{|s|}{2}},\]
and it holds that $z_{(\pm)}\,\varphi, \, M\,\varphi \in \Lsp{1}(\bR)$.

\item[$(C_3)$] It is clear that $f(\cdot,y)$ is measurable for each fixed $y\in\bR$ and $f(t,\cdot)$ is continuous for a.\,e. $t\in \bR$. Finally, for each $r>0$, there exists $\phi_{r}(t)=\frac{\sqrt{r}\,\sin^2 t}{\sqrt{|t|}} \in \Lsp{\infty}(\bR)$ such that \[\frac{f(t,x\varphi(t))}{\varphi(t)}=\frac{\sqrt{|x\,t|}\,\sin^2 t} {|t|}\le \phi_{r}(t),\] for all $x\in [-r,r]$ and a.\,e. $t\in \bR$.

\item[$(C_4)$] In this case, \[\alpha(|k(\cdot,s)|)=\min_{t\in\left[\frac{\pi}{4},\frac{3\,\pi}{4}\right]} |k(t,s)|-\frac{\sqrt{2}}{2}\,\|k(\cdot,s)\|_\infty = e^{-\frac{|s|}{2}}\,\min_{t\in\left[\frac{\pi}{4},\frac{3\,\pi}{4}\right]} |\sin t|-\frac{\sqrt{2}}{2}\,e^{-\frac{|s|}{2}} =0. \]

\item[$(C_5)$] It is clear that $\alpha(|k(t,s)|)\,\varphi(s)\in \Lsp{1}(\bR)$. Moreover, for all $u\in P$, it holds that
\begin{equation*}\begin{split} 
\alpha(L_1 u) =&\min_{t\in\left[\frac{\pi}{4},\frac{3\,\pi}{4}\right]} \int_{-\infty}^{\infty}|k(t,s)|\,u(s)\,\dif s - \frac{\sqrt{2}}{2} \, \left\| \int_{-\infty}^{\infty}|k(t,s)|\,u(s)\,\dif s\right\|_\infty \\
&\, \ge \int_{-\infty}^{\infty}\min_{t\in\left[\frac{\pi}{4},\frac{3\,\pi}{4}\right]} |k(t,s)|\,u(s)\,\dif s -\frac{\sqrt{2}}{2} \, \int_{-\infty}^{\infty}\|k(t,s)\|_\infty\,u(s)\,\dif s \\ & \,=\int_{-\infty}^{\infty}\alpha(|k(\cdot,s)|)\,u(s)\,\dif s.
\end{split}\end{equation*}

\item[$(C_6)$] We can take $A=\left[\frac{\pi}{4},\frac{3\,\pi}{4}\right]$. For such $A$, we obtain
\[\frac{1}{\widetilde{M}(A)}=\inf_{t\in A}\int_{A} e^{-\frac{|s|}{2}}\, \sin t \dif s = \inf_{t\in A} \, \left\{ 2\,e^{\frac{-3\,\pi}{8}} \,\left(-1+e^{\frac{\pi}{4}}\right)\, \sin t\right\} = \sqrt{2}\,e^{\frac{-3\,\pi}{8}} \,\left(-1+e^{\frac{\pi}{4}}\right)>0.\]

\item[$(C_7)$]  It is analogous to $(C_4)$. The same occurs to $(C_9)$.

\item[$(C_8)$] It is analogous to $(C_5)$. The same occurs to $(C_{10})$.

Finally, we obtain the following values for the limits $f^\infty$ and $f_0$:
\[f^\infty=\varlimsup_{|x| \to+\infty}\frac{\sup\limits_{t \in \bR} \, \dfrac{\sqrt{|x|} \, \sin^2 t}{\sqrt{|t|}}}{|x|} \le \varlimsup_{|x| \to+\infty} \frac{\sqrt{|x|}}{|x|}=0, \]
and so $f^\infty=0$. Analogously,
\[f_0=\varliminf_{|x| \to 0}\frac{\inf\limits_{t \in A} \, \dfrac{\sqrt{|x|} \,\sin^2 t}{\sqrt{|t|}}}{|x|}= \varliminf_{|x| \to 0} \frac{\sqrt{|x|}}{\sqrt{3\,\pi}\,|x|}=+\infty.\]

On the other hand, since both $r(L_1)$ and $r(L_2)$ are positive (as it has been proved in Theorem \ref{specrad}), it holds that $\mu(L_1)>0$ and $\mu(L_2)<+\infty$.

Thus, from $(T_2)$ in Theorem \ref{thmones}, we deduce that our problem has at least a non-trivial solution in $K_\alpha \subset \cC_\varphi$.
\end{itemize}

\end{document}